\documentclass[a4paper]{amsart}
\usepackage[
left=28truemm,
right=28truemm
]{geometry}
\usepackage{amsfonts, amssymb, amsmath, verbatim, amsthm, mathrsfs, amscd, enumerate, ascmac, fancyhdr, caption, hyperref, framed, subfigure}
\usepackage{bbm}
\numberwithin{equation}{section}
\usepackage{tikz}
\usetikzlibrary{positioning,intersections, calc, arrows,decorations.markings, angles}
\theoremstyle{plain}
\newtheorem{theorem}{Theorem}[section]
\newtheorem{lemma}[theorem]{Lemma}
\newtheorem*{lemma*}{Lemma}
\newtheorem{proposition}[theorem]{Proposition}
\newtheorem*{proposition*}{Proposition}

\theoremstyle{definition}

\newtheorem*{claim*}{Claim}
\theoremstyle{remark}
\newtheorem*{remarks}{Remarks}
\newtheorem*{remark}{Remark}

\def\supp{\mathrm{supp}\,}
\def\d{\mathrm{d}}
\def\ae{\textrm{a.e.}}

\def\C{\mathfrak C}

\def\D{\mathfrak{D}}
\def\I{\mathcal I}

\author{Chu-hee Cho}
\address[Chu-hee Cho] {Department of Mathematical Sciences and RIM, Seoul National University, Seoul 08826, Republic of Korea}
\email{akilus@snu.ac.kr}

\author{Shobu Shiraki}
\address[Shobu Shiraki] {Department of Mathematics, Graduate school of Science and Engineering, Saitama University, Saitama, 338-8570, Japan}
\email{sshiraki@mail.saitama-u.ac.jp}

\thanks{This work was supported by 
NRF grant no. 2020R1I1A1A01072942, 2022R1A4A1018904 (Republic of Korea) (Cho)
and 
JSPS KAKENHI Grant-in-Aid for JSPS Fellows no.  20J11851 (Shiraki).
}

\begin{document}
\date{\today}
\title[
Proceeding for RIMS
]
{
A note on some variations of the maximal inequality for the fractional Schr\"odinger equation
}

\subjclass[2010]{}
\begin{abstract}
The purpose of this note is to provide a summary of the recent work of the authors on two variations of the pointwise convergence problem for the solutions to the fractional Schr\"odinger equations; convergence along a tangential line and along a set of lines, as exhibiting some new results in each setting.
For the former case, we make a simple observation on a path along a tangential curve of exponential order.
We discuss counterexamples for the latter case that show some of the known smooth regularities are essentially optimal.


\end{abstract}
\maketitle
\section{Introduction}

Let $d\in\mathbb N$ and $m>1$. 
On $\mathbb R^d\times \mathbb R$ the fractional Schr\"odinger equation is famously known as
\begin{equation}\label{e:Scrhodinger}
i\partial_t u + (-\Delta)^{\frac{m}2} u = 0
\end{equation}
for the initial data $u(\cdot,0) = f$, whose solution may be expressed (at least formally) as
\[
u(x,t)
=
S_t^mf(x)
=
\left(2\pi\right)^{-d} \int_{\mathbb R^d} e^{i(x\cdot \xi + t|\xi|^m)} \widehat{f}(\xi)\,\d\xi
\]
by using the Fourier transform given by $\widehat{f}(\xi)=\int_{\mathbb R^d} e^{-ix\cdot\xi} f(x)\,\d x$. 

The main object of our interest in this note is to determine the optimal regularity $s\geq0$ for which  the following local maximal-in-time estimate with respect  to time for the fractional Schr\"odinger equation holds; for some $q\geq1$, there exists a constant $C>0$ such that
\begin{equation}\label{i:max classical}
\|S_t^m f(x)\|_{L_x^q(\mathbb B^d)L_t^\infty(\mathbb I)}
\leq
C
\|f\|_{H^s(\mathbb R^d)}
\end{equation}
for all $f \in H^s(\mathbb R^d)$, defined by
\[
\|f\|_{H^s(\mathbb{R}^d)}=\|(1-\Delta)^\frac s2f\|_{L^2(\mathbb{R}^d)}=(2\pi)^{-\frac d2}\left(\int_{\mathbb R^d}(1+|\xi|^2)^s|\widehat{f}(\xi)|^2\,\d\xi\right)^\frac12.
\]
By locality, once we prove \eqref{i:max classical} for some $q_0\geq1$,  the inequality \eqref{i:max classical} (under the same conditions but) with $q$ smaller than $q_0$ is deduced by H\"older's inequality.
Of course, when the initial data is smooth enough, for instance, strictly smoother than $\frac d2$, the validity of \eqref{i:max classical} (with arbitrary $q$) follows immediately. In fact, a trivial computation reveals that for any $(x,t)\in \mathbb B^d\times \mathbb I$
\begin{align*}
|e^{it(-\Delta)^\frac m2}f(x)|
\lesssim
\int_{\mathbb R^d} |\widehat{f}(\xi)|\,\d\xi
\lesssim \left(\int_\mathbb R(1+r^2)^{-s}r^{d-1}\,\d r\right)^\frac12\|f\|_{H^s(\mathbb R^d)},
\end{align*}
which is finite whenever $s>\frac d2$. 
Moreover, this computation indicates that the oscillatory cancellation, completely ignored in the first step, may have a crucial role in order to go beyond the smooth regularity.

The maximal inequality \eqref{i:max classical} is motivated by the study of the pointwise convergence behavior of the solution to the fractional Schr\"odinger equation, sometimes referred to as Carleson's problem.  
Namely, if there exists $C>0$ such that \eqref{i:max classical} holds\footnote{The left-hand side of \eqref{i:max classical} can be weakened by replacing $L_x^1(\mathbb B^d)$ by the weak-type space $L_x^{q,\infty}(\mathbb B^d)$.}
for all $f\in H^s$ for some $q\geq1$ and $s\geq0$, then 
\begin{equation}\label{pw}
\lim_{t\to0}S_t^m f(x)=f(x)\quad \ae
\end{equation}
holds for all $f\in H^s$. 
The reduction is a similar spirit to the Lebesgue differentiation theorem (\cite{Duo}) and it does not lose much information. In fact, \eqref{i:max classical} conversely follows from the pointwise convergence \eqref{pw} provided the \emph{weak-type} maximal estimate for $q\in[1,2]$ via Nikishin--Stein maximal principle (see, for example, \cite{Pierce} for the details).
When the spatial dimension $d=1$, the problem is relatively easy and was completely solved in the 1980s.
\begin{theorem}[Carleson \cite{Cr80}, Dalberg--Kenig \cite{DK82}, Sj\"olin \cite{Sj87}, Kenig--Ponce--Vega \cite{KPV91}]\label{t:classical d=1}
Let $d=1$ and $m>1$.  Then, there exists $C>0$ such that \eqref{i:max classical} holds with $q=4$ for all $f\in H^s(\mathbb R)$ if and only if $s\geq\frac14$.
\end{theorem}
In higher dimensions, it turned out to be extremely difficult and one can find some historical contributions in \cite{Br95, Br13, Lee06, LR17, MYZ15, Vg88} for example. The breakthrough came with Bourgain's number theoretic counterexample for the standard Schr\"odinger equation in 2016 \cite{Br16} (see also an expository summary of his argument \cite{Pierce} due to Pierce and an alternative proof \cite{LR19} due to Luc\`{a}--Rogers) Thanks to the strong connections of Carleson's problem with other crucial conjectures such as Stein's restriction conjecture and Kakeya maximal conjecture, soon later Du--Guth--Li \cite{DGL17} and Du--Zhang \cite{DZ18} applied state-of-the-art tools in harmonic analysis, such as multilinear restriction theorem, decoupling inequality, polynomial partitioning, refined Strichartz estimates, and showed the necessary regularity given by Bougain is essentially sufficient for the maximal estimate \eqref{pw} (except the endpoint).

\begin{theorem}[Bourgain \cite{Br16}, Du--Guth--Li \cite{DGL17}, Du--Zhang \cite{DZ18}, Cho--Ko \cite{CK18}]\label{t:classical d}
Let $d\geq2$ and $m>1$.
For $d=2$, \eqref{i:max classical} with $q=3$ holds if $s>\frac13$ and, for $d\geq3$, \eqref{i:max classical} with $q=2$ holds if $s>\frac{d}{2(d+1)}$. Moreover, $s\geq\frac{d}{2(d+1)}$ is necessary for \eqref{i:max classical} with $m=2$ and $q\geq1$.
\end{theorem}

It is still open whether $s\geq\frac{d}{2(d+1)}$ is necessary in the case of the fractional Schr\"odinger equation. See the recent progress by An--Chu--Pierce \cite{ACP21} and Eceizabarrena--Ponce-Vanegas \cite{EP22} in this direction.\\

Throughout, the unit ball centered at the origin in $\mathbb R^d$ is denoted by $\mathbb B^d$ and $\mathbb I:=\mathbb B^1$ to emphasize that it coincides with the unit interval.

\section{Fractal dimension of the divergence sets}

Although we know that the Lebesgue measure of the \emph{divergence set} $\D(f):=\{x\in\mathbb R: S_t^m f(x)\not\to f(x)\ \text{as $t\to0$} \}$ is zero for $f\in H^s$ with $s\geq\frac14$ as a consequence of Theorem \ref{t:classical d=1}, the set $\D(f)$ may still be large enough to be ``detected" by a fractal measure. This direction in the context of pointwise convergence was first considered by Sj\"ogren--Sj\"olin \cite{SS89}. In \cite{BBCR11} Barcel\'{o}--Bennett--Carbery--Rogers recently concerned with this question and measured the divergence set by the use of Frostman's lemma, together with the results about the singularities of the Bessel potential due to \v{Z}ubrini\'{c}.

\begin{theorem}[Barcel\'{o}--Bennett--Carbery--Rogers \cite{BBCR11},  \v{Z}ubrini\'{c} \cite{Zb02}]
Let $d=1$ and $m>1$. Then,
\[
\sup_{f\in H^s}\dim_H\D(f)=1-2s,
\]
where $\dim_H$ denotes the Hausdorff dimension.
\end{theorem}
The analogous results in higher dimensions are also available but there are still some remaining unknown cases for small $s$. Interested readers are encouraged to visit \cite{DGLZ18, LP21, LR17, LR19b} aside from \cite{BBCR11}.  The key estimate is the maximal inequality \eqref{i:max classical} yet its spatial measure $\d x$ is replaced by the $\alpha$-dimensional measure $\d\mu$ characterized by the property
\[
\sup_{x\in\mathbb R^d,r>0}\frac{\mu(B^d(x,r))}{r^\alpha}<\infty,
\]
where $0<\alpha\leq d$ and a ball $B^d(x,r)$ of radius $r$ centered at $x$.

\section{Convergence along a tangential curve}
The original convergence of the solutions to the fractional Schr\"odinger equations can be regarded as the limit along the vertical line to the hyperplane $\mathbb R^d\times\{0\}$ at $x$, i.e. 
\[
\lim_{t\to0}S_t^mf(x)=\lim_{\substack{(y,t)\to(x,0)\\ (y,t)\in\{x\}\times (0,1)}}S_t^mf(y).
\]
We shall replace the path of the vertical line with more general paths. In this section, let us consider a convergence along a curve. When $d=1$, we shall define curves by 
\[
\rho_\kappa(x,t)=x-t^\kappa
\]
with $\kappa>0$ and call it \emph{non-tangential} and \emph{tangential} when $\kappa\geq1$ and $0<\kappa<1$, respectively. The corresponding pointwise convergence problem along a tangential curve is 
\begin{equation}\label{pw tangential}
\lim_{\substack{(y,t)\to(x,0)\\ y=x-t^\kappa}} S_t^mf(y)
=
\lim_{t\to0}S_t^mf(\rho_\kappa(x,t))
=
f(x)\quad \ae \,\, x.
\end{equation}
By the standard argument mentioned earlier (see also \cite{CS20} in this particular setting), it suffices to show the maximal estimate: 
\begin{equation}\label{i:maximal curve}
\|S_t^m f(\rho_\kappa(x,t))\|_{L_x^q(\mathbb I)L_t^\infty(\mathbb I)}\leq C\|f\|_{H^s}
\end{equation}
for all $f\in H^s$. In the study of pointwise convergence with harmonic operators Lee--Rogers \cite{LR12} discovered\footnote{They only dealt with the standard Schr\"odinger equation but the same conclusion holds for $m>1$. In order to show this, a similar argument in \cite{CS20} can be carried since $|t_1-t_2|\gtrsim |t_1^\kappa-t_2^\kappa|$ by the mean value theorem.} that 
there is no difference between the pointwise convergence along a non-tangential curve \eqref{pw tangential} and the one along the vertical line \eqref{i:max classical}. Intuitively, this makes sense since the non-tangential curve literally looks like the vertical line in a neighborhood of $(x,0)$ in the space-time plane for sufficiently large enough $\kappa$.
%
\begin{theorem}[Lee--Rogers \cite{LR12}]
Let $d=1$, $m>1$ and $\kappa\geq1$. Then, \eqref{i:maximal curve} with $q=2$ holds for all $f\in H^s$ if $s\geq\frac14$.
\end{theorem}
On the other hand, when the curve $\rho_\kappa(x,t)$ is tangential ($0<\kappa<1$), this may not be true anymore. 

\begin{theorem}[Cho--Lee--Vargas \cite{CLV12}, Cho--Lee \cite{CL14}, Cho--Shiraki  \cite{CS20}]\label{t:max tangential}
Let $d=1$, $m>1$, $0<\kappa\leq1$ and $\mu$ be an $\alpha$-dimensional measure. If $s>\max\{\frac14,\frac{1-\alpha}{2},\frac{1-m\alpha\kappa}{2}\}$, then there exists $C>0$ such that 
\begin{equation}\label{i:max tangential}
\|S_t^m f(\rho_\kappa(x,t))\|_{L_x^2(\mathbb I,\d\mu)L_t^\infty(\mathbb I)}
\leq 
C
\|f\|_{H^s(\mathbb R)}
\end{equation}
for all $f\in H^s(\mathbb R)$.
\end{theorem}

Here, the condition $s>\max\{\frac14,\frac{1-\alpha}{2},\frac{1-m\alpha\kappa}{2}\}$ is sharp in the sense that one can find an $\alpha$-dimensional measure $\mu$ and the initial data $f$ such that \eqref{i:max tangential} fails whenever $s<\max\{\frac14,\frac{1-\alpha}{2},\frac{1-m\alpha\kappa}{2}\}$. As we see momentarily,  this is based on the Knapp-type argument where we also restrict the domain of $(x,t)\in\mathbb B^n\times \mathbb I$.  For instance, the authors in \cite{CS20} set $\d\mu(x)=|x|^{\alpha-1}dx$ and for each condition\footnote{The same spirit for the second row of the table is effectively used in our Proposition \ref{p:infty tangential} and Theorem \ref{t:nec fractal path} in the subsequent sections. }, choose the initial data $f$ and the restriction of $(x,t)$ as follows for $\lambda\geq1$, $A=[0,\lambda^\frac1m]$ and $B=[\lambda,\lambda+\lambda^{-1}]$:

\renewcommand{\arraystretch}{1.5}
\begin{table}[htb]
\centering
 \begin{tabular}{|c||c|c|c|}  \hline 
    To show &
    Initial data &
    $x$ &
    $t=t(x)$\\
    \hline \hline
    $s\geq\frac14$ &
    $\widehat{f}(\xi)=\lambda^{-1}\chi_B(\xi)$ &
    $(0, \frac{1}{100(m-1)})$ &
    $\text{$t(x)$ s.t. $x=t(x)^\kappa+m\lambda^{2m-2}t(x)$}$\\
    \hline
    $s\geq\frac{1-\alpha}{2}$ &
    $\widehat{f}(\xi)=\chi_A(\xi)$ &
    $(0,\frac{1}{100}\lambda^{-\frac1m})$ &
    $t(x)\in (0,\frac{1}{100}\lambda^{-1})$ \\ 
    \hline
    $s\geq\frac{1-m\alpha\kappa}{2}$ &
    $\widehat{f}(\xi)=\chi_A(\xi)$ &
    $(0,\frac{1}{100}\lambda^{-\frac1m})$&
    $t(x)=x^{\frac1\kappa}$\\
    \hline
 \end{tabular}
\end{table}
By applying a similar argument in the classical situation, Theorem \ref{t:max tangential} implies that the pointwise convergence along the tangential curve \eqref{pw tangential} holds for all $f\in H^s$ if $s>\max\{\frac14,\frac{1-m\kappa}{2}\}$, and moreover, $\sup_{f\in H^s}\dim_H\D(f\circ \rho_\kappa)\leq\max\{1-2s,\frac{1-2s}{m\kappa}\}$ if $\max\{\frac14,\frac{1-m\kappa}{2}\} < s < \frac12$.
It is worth noting that the upper bound of the Hausdorff dimension of the divergence sets varies depending on $\kappa\in (0,1]$: The curve $\rho_\kappa$ is classified the same as the vertical line when $\kappa\in(\frac1m,1]$. When $\kappa<\frac1m$, the number $\frac{1-2s}{m\kappa}$ is dominant over $1-2s$. In particular, for $\kappa\in(\frac{1}{2m},\frac1m)$, the gap at $s=\frac14$ for the upper bound of $\dim_H\mathfrak D(f\circ\rho_\kappa)$ remains in existence but reflects the mixture state illustrated in Figure \ref{f:kappa-s} despite the fact that $\frac14\geq\frac{1-m\kappa}{2}$. The gap disappears when $\kappa\in(0,\frac{1}{2m})$.
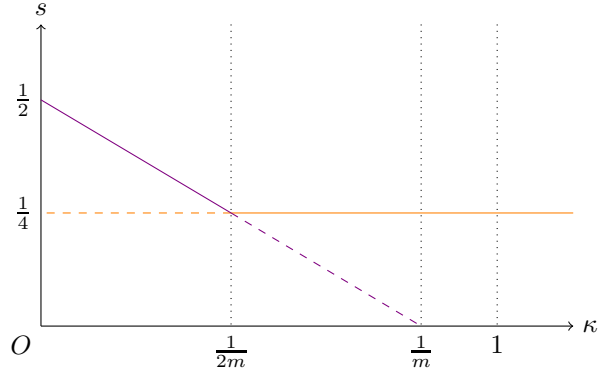
\begin{figure}[h]
\begin{center}
\begin{tikzpicture}[scale=1]
\node [below left] at (0,0) {$O$};

\draw [->] (0,0)--(7,0) node[right] {$\kappa$};
\draw [->] (0,0)--(0,4) node [above] {$s$};
\draw [dotted](6,4)--(6,0) node [below]   {$1$};
\draw [dotted](5,4)--(5,0) node [below] {$\frac1m$};
\draw [dotted](2.5,4)--(2.5,0) node [below] {$\frac{1}{2m}$};

\node at (0,3) [left] {$\frac12$};
\draw [color=orange] (7,1.5)--(2.5,1.5);
\draw [color=orange,dashed] (2.5,1.5)--(0,1.5) node [color=black,left] {$\frac14$};

\draw [color=blue!50!red!,dashed] (2.5,1.5)--(5,0);
\draw [color=blue!50!red!] (0,3)--(2.5,1.5);

\end{tikzpicture}
\caption{The sharp smooth regularity depending on $\kappa$ when $m>1$.}\label{f:kappa-s}
\end{center}
\end{figure}

In the higher dimensional cases of \eqref{i:maximal curve}, the formulation of the curve is rather abstract. 
Recently Li--Wang  \cite{LW18, LW21} obtained some partial results for a curve $\gamma$ such that 
\[
\begin{cases}
|\gamma(x,t)-\gamma(x,t')|\leq c|t-t'|^\kappa\\
\gamma(x,0)=x
\end{cases}
\]
for $x\in\mathbb B^d$, $t,t'\in \mathbb I$, $\kappa\in (0,1)$ and for some $c>0$.

Casually speaking, the smaller the H\"older continuity index becomes, the more the curve $\rho_\kappa(x,t)$ gets ``tangent" to the hyperplane. This may be clearer if we write $y=-\rho_\kappa(x,t)=t^\kappa-x$ and re-express it as $t=(y+x)^{\frac1\kappa}$, the graph of $y\mapsto t$ touching $\mathbb R\times\{0\}$ at $y=-x$. While the curve $\rho_\kappa(x,t)$ is of polynomial tangential (of order $\frac1\kappa$), one may wonder what happens for convergence along a curve of order $\infty$, or exponentially tangent curve (beyond polynomial order).
The typical example of such curve is formed by $t=e^{-\frac{1}{y+x}}$, instead of $t=(y+x)^{\frac1\kappa}$. Considering the convergence along this ``exponentially tangential" curve, which is reformulated and denoted as $\widetilde{\rho}(x,t)=x-(\log\frac1t)^{-1}$, one can show that the smooth regularity for the corresponding maximal inequality is, consistently, as almost bad as a trivial result; $s\geq\frac12$. 
\begin{proposition}\label{p:infty tangential}
Let $m>1$, $0<\kappa\leq1$ and $\d\mu(x)=|x|^{\alpha-1}\d x$. Then there exists $C>0$ such that 
\begin{equation}\label{maximal endpoint}
\|S_t^m f(\widetilde{\rho}(x,t))\|_{L_x^2(\mathbb I,\d\mu)L_t^\infty(\mathbb I)}\leq C\|f\|_{H^s(\mathbb R)}
\end{equation}
fails if $s<\frac12$.
\end{proposition}

\begin{proof}

Let $s<\frac12$ and suppose \eqref{maximal endpoint} held. Define the initial data $f$ by
\[
\widehat{f}(\xi)=\chi_{A}(\xi), \quad A=[0,\tfrac{1}{100}\lambda^\frac1m]
\]
so that $\|f\|_{H^s(\mathbb R)}\lesssim\lambda^{\frac sm}\lambda^{\frac{1}{2m}}$. Then, considering $t=t(x)$ as a fuction of $x$, it holds that
\[
\sup_{t\in \mathbb I}|S_t^m f(\widetilde{\rho}(x,t))|
\ge
\left|\int_{A} e^{i(x-(\log1/t(x))^{-1})\xi+t(x)|\xi|^m)}\,\d\xi\right|.
\]
For each of the choices of $x\in(0,(\log\lambda)^{-1})$ and $t=t(x)=e^{-\frac1x}$ let the phase be fairly small, namely,
\[
|(x-(\log1/t(x))^{-1})\xi+t(x)|\xi|^m|\le \frac12
\]
for a sufficiently large $\lambda$.
Therefore, we can choose $\varepsilon$ so that $0<\varepsilon<\frac12-s$ and
\begin{align*}
\Big\|\sup_{t\in \mathbb I}\big|S_t^m f(\widetilde{\rho}(x,t))\big|\Big\|_{L^2(\mathbb I,\d\mu)}
&\gtrsim
\|S_t^m f(\widetilde{\rho}(x,e^{-\frac1x}))\|_{L^2((0,(\log\lambda)^{-1}),\d\mu)}\\
&\gtrsim
\lambda^{\frac1m}(\log \lambda)^{-\frac\alpha2}\\
&\gtrsim
\lambda^{\frac1m-\frac\varepsilon m}
\end{align*}
for such large number $\lambda$.
Combined with the estimate of the right-hand side, it follows that 
\[
\lambda^{\frac1m-\frac\varepsilon m}\lesssim\lambda^{\frac{s}{m}+\frac{1}{2m}},
\]
which is a contradiction as $\lambda\to\infty$ under $s<\frac12$.
\end{proof}

\section{Convergence along a set of non-tangential lines}\label{s:fractal}
In \cite{SS89}, Sj\"ogren--Sj\"olin (see \cite{Jo10} for the fractional Schr\"odinger equation) considered the convergence within a conical region over $(x,0)\in\mathbb R^d\times\{0\}$, instead of the limit along the vertical line to the point $(x,0)$, and proved that the trivial regularity $s>\frac d2$ (as we  observed at the beginning of this note) is actually necessary in this case:
It is tempting to unify their result and Theorem \ref{t:classical d=1}/\ref{t:classical d}.
To do so, notice that the vertical line is regarded as a line $\{x+t\theta:x\in \mathbb R^d,\ t\in \mathbb I, \ \theta\in \{0\}\}$, while the conical region is regarded as a set $\{x+t\theta:x\in\mathbb R^d,\ t\in \mathbb I, \ \theta\in [-1,1]\}$ for example.
In the one-dimensional case, Lee--Vargas and the first author considered convergence along any path within a region generated by a set $\{x+t\theta:x\in\mathbb R^d, \ t\in \mathbb I,  \ \theta\in \Theta\}$ for a given compact set $\Theta\subset\mathbb R$ whose Minkowski dimension is lied in $(0,1)$, such as the third Cantor set. This is an intermediate case since the Minkowski dimension of $\{0\}$ and $[-1,1]$ are $0$ and $1$, respectively. In general, the Minkowski dimension of a compact subset $\Theta\subset \mathbb R^d$ is defined by
\[
\dim_M\Theta
=
\inf\{\beta>0: \limsup_{\delta\to0}N_\delta(\Theta)\delta^\beta=0\}
\]
for $N_\delta(\Theta)$ denoting the smallest number of $\delta$-ball covering of $\Theta$.
By letting $\Theta$ generate the path (the conical region with a bunch of linear holes in it), the following sufficient result has been obtained. We give the path explicitly by
\[
\varrho(x,t,\theta)=x+t\theta, \quad (x,t,\theta)\in\mathbb B^d\times\mathbb I\times \Theta
\]
and let $\beta(\Theta)=\dim_M\Theta$.

\begin{theorem}[Cho--Lee--Vargas \cite{CLV12}, Shiraki \cite{Sh19}]\label{t:max fractal}
Let $d=1$, $m>1$ and $\Theta$ be a compact subset of $\mathbb R$.  If $s>\frac{1+\beta(\Theta)}{4}$, then 
\begin{equation}\label{i:max fractal 1}
\|S_t^mf(\varrho(x,t,\theta))\|_{L_x^4(\mathbb I)L_t^\infty(\mathbb I)L_\theta^\infty(\Theta)}
\lesssim
\|f\|_{H^s(\mathbb R)}
\end{equation}
for all $f\in H^s(\mathbb R)$.
\end{theorem}

This result immediately implies that the pointwise convergence along non-tangential lines 
\[
\lim_{t\to0}S_t^m f(\varrho(x,t,\theta))=f(x)\ \ae
\]
holds for all $f\in H^s(\mathbb R)$ with $s>\frac14+\frac{\beta(\Theta)}{4}$.  Very recently the higher dimensional cases are also considered.  Li--Wang--Yan \cite{LWY20} adapted an analogous reduction argument in \cite{LR12} and invoked the results for the pointwise convergence along the vertical line. In particular, when $d=2$, combining with the result from \cite{DGL17}, they showed that 
there exists $C>0$ such that 
\begin{equation}\label{i:max fractal 2}
\|S_t^2f(\varrho(x,t,\theta))\|_{L_x^3(\mathbb B^2)L_t^\infty(\mathbb I)L_\theta^\infty(\Theta)}
\lesssim
\|f\|_{H^s(\mathbb R^2)}
\end{equation}
for all $f\in H^s(\mathbb R^2)$ whenever $s>\frac{1+\beta(\Theta)}{3}$, which interpolates the case $s>\frac13$ for $\beta(\Theta)=0$ \cite{DGL17} and $s>1$ for $\beta(\Theta)=2$ \cite{SS89}.\\

As far as the authors are aware, there was no result that indicates whether or not the regularity $s>\frac{d}{2(d+1)}(1+\beta(\Theta))$ is sharp for any $d$, unless $\Theta$ has either $0$ or the full dimension. We construct a counterexample that shows that $s>\frac{1+\beta(\Theta)}{4}$ for \eqref{i:max fractal 1} and $s>\frac{1+\beta(\Theta)}{3}$ for $\eqref{i:max fractal 2}$ are reasonable in the case of $\d\mu(x)=\d x$.
\begin{theorem}\label{t:nec fractal path}
Let $d\geq1$ and $m>1$. Then, there exists $\Theta\subset\mathbb R^d$ such that 
\[
\|S_t^mf(\varrho(x,t,\theta))\|_{L^q_x(\mathbb B^d)L^\infty_t(\mathbb I)L^\infty_\theta(\Theta)}
\leq
C\|f\|_{H^s(\mathbb R^d)}
\]
fails if $s<\frac d2-\frac dq+\frac{\beta(\Theta)}{q}$.  
\end{theorem}
\begin{remarks}
\begin{enumerate}[(i)]
\item As alluded to earlier, Theorem \ref{t:nec fractal path} in the case when $(d,q)=(1,4)$ and $m>1$ shows that \eqref{i:max fractal 1} fails if $s<\frac{1+\beta(\Theta)}{4}$, and the case when $(d,q,m)=(2,3,2)$ shows that  \eqref{i:max fractal 2} fails if $s<\frac{1+\beta(\Theta)}{3}$.

\item Since we do not know whether the standard step to deduce (pointwise convergence) $\Rightarrow$ (maximal estimate) by Stein's maximal principle in this variant of convergence, there is, unfortunately, no conclusion for pointwise convergence result from Theorem \ref{t:nec fractal path}. (If we assume that Stein's maximal principle was carried in this setting with a set of fractal lines, the valid range of the exponent $q$ might be $1\leq q\leq2$ anyway.)


\end{enumerate}
\end{remarks}

\begin{proof}
For a fixed $r\in (0,\frac12)$ define the $r$-Cantor set $\C(r)$ by taking the intersection all generations of the pre-Cantor sets $\C_k(r)$ for each non-negative $k\in \mathbb Z$ (i.e. $\C(r)=\bigcap_{k=0}^\infty\C_k(r)$), where 
$\C_k(r)$ are inductively generated as follows: Starting with $\C_0(r)=[0,1]$, we remove the interval of length $1-2r$ from the middle of $[0,1]$ and denote the remaining $2$ intervals together by $\C_1(r)$. Similarly, we remove the interval of length $r(1-2r)$ from the middle of each interval of $\C_1(r)$ and denote the remaining $2^2$ intervals together by $\C_2(r)$, and so on. Note that, by following the construction, $\C_k(r)$ consists of disjoint $2^k$ intervals of length $r^k$, each of which we let $\Omega_{k,j}$ so that $\C_k(r)=\bigcup_{j=1}^{2^k} \Omega_{k,j}$ ($|\Omega_{k,j}|=r^{k}$) and $\C_k(r)\supset\C_{k+1}(r)$. One of crucial properties of $\C(r)$ in this context is $\dim_M\C(r)=\frac{\log 2}{\log1/r}\in(0,1)$ (Appendix \ref{a:Cantor sets}).

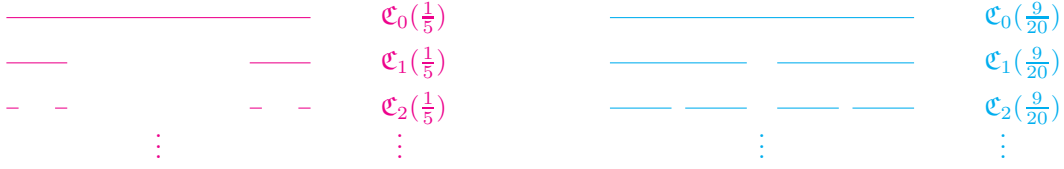
\begin{figure}[h]
\begin{center}
\begin{minipage}{.48\linewidth}
\centering
\begin{tikzpicture}[scale=4]
\textcolor{magenta}{
\draw (0,1.65)--(1,1.65);
\draw (0,1.5)--(0.2,1.5);
\draw (1-0.2,1.5)--(1,1.5);
\draw (0,1.35)--(0.2^2,1.35);
\draw (0.2-0.2^2,1.35)--(0.2,1.35);
\draw (1-0.2,1.35)--(1-0.2+0.2^2,1.35);
\draw (1-0.2^2,1.35)--(1,1.35);
\node at (0.5,1.25) {$\vdots$};
\node [right] at (1.2,1.65) {$\C_0(\frac15)$};
\node [right] at (1.2,1.5) {$\C_1(\frac15)$};
\node [right] at (1.2,1.35) {$\C_2(\frac15)$};
\node [right] at (1.25,1.25) {$\vdots$}; 
}
\end{tikzpicture}

\end{minipage}
\hfill
\begin{minipage}{.48\linewidth}
\centering
\begin{tikzpicture}[scale=4]
\textcolor{cyan}{
\draw (0,1.65)--(1,1.65);
\draw (0,1.5)--(0.45,1.5);
\draw (1-0.45,1.5)--(1,1.5);
\draw (0,1.35)--(0.45^2,1.35);
\draw (0.45-0.45^2,1.35)--(0.45,1.35);
\draw (1-0.45,1.35)--(1-0.45+0.45^2,1.35);
\draw (1-0.45^2,1.35)--(1,1.35);
\node at (0.5,1.25) {$\vdots$};
\node [right] at (1.2,1.65) {$\C_0(\frac{9}{20})$};
\node [right] at (1.2,1.5) {$\C_1(\frac{9}{20})$};
\node [right] at (1.2,1.35) {$\C_2(\frac{9}{20})$};
\node [right] at (1.25,1.25) {$\vdots$}; 
}

\end{tikzpicture}

\end{minipage}

\end{center}
\caption{Each generation of pre-Cantor sets associated with \textcolor{magenta}{$\C(\frac15)$} and \textcolor{cyan}{$\C(\frac{9}{20})$}.}
\end{figure}

For sufficiently large $k\gg1$, let $\lambda=\lambda_k:=r^{-k}$. Let $s<\frac d2-\frac dq+\frac{\beta(\Theta)}{q}$ and suppose that the stated inequality held. In order to see the main idea (based on the choice in the second row of the table on page 4), first let us deal with the case when $d=1$. Set $\Theta=\C(r)$ and the initial data $f$ by
\begin{equation}\label{e:initial necessary}
\widehat{f}(\xi)
=
e^{-i|\xi|^m}\chi_{D_1}(\xi),\quad D_1=[0,c\lambda]
\end{equation}
so that $\beta(\Theta)=\frac{\log 2}{\log1/r}\in(0,1)$ and $\|f\|_{H^s(\mathbb R)}\lesssim \lambda^{s}|\,D_1|^{\frac12}=\lambda^{s+\frac12}$. By the change of variables; $x=-y$ and $t=1-\tau$, it would follow that 
\begin{align*}
\sup_{\substack{t\in \mathbb I\\\theta\in\Theta}}\big|S_t^m f(\varrho(-y,t,\theta))\big|
&=
\sup_{\substack{t\in \mathbb I\\\theta\in\Theta}}
\left|
\int_{D_1} e^{i((-y+t\theta)\xi+t|\xi|^m)}e^{-i|\xi|^m)}\,\d\xi
\right|\\
&\geq
\sup_{\substack{\tau\in [0,1]\\\theta\in\Theta}}
\left|
\int_{D_1} e^{i(-(y-\theta)\xi-\tau\theta\xi-\tau|\xi|^m)}\,\d\xi
\right|\\
&\geq
\left|
\int_{D_1} e^{i(-(y-\theta(y))\xi-\tau(y)\theta(y)\xi-\tau(y)|\xi|^m)}\,\d\xi
\right|.
\end{align*}
Now, to ensure the phase is fairly small, we specify $(x,t)$. If we let $y\in \C_k(r)$ then we can find $\theta(y)\in\Theta$ satisfying $|y-\theta(y)|<\lambda^{-1}$ (for instance, take the endpoints of each interval $\Omega_{k,j}$, which are also in $\Theta$). Hence, for $\tau\in(0,\lambda^{-m})$, $y\in\C_k(r)$ and such $\theta(y)\in\Theta$, the phase is bounded above by
\[
|(y-\theta(y))\xi+\tau\theta(y)\xi+\tau|\xi|^m|
\leq\frac12,
\]
which implies that 
\begin{align*}
\Big\|\sup_{\substack{t\in \mathbb I\\\theta\in\Theta}}\big|S_t^mf(\varrho(x,t,\theta))\big|\Big\|_{L^q(\mathbb I)}
&\gtrsim
\Big\|\sup_{\substack{\tau\in (0,\lambda^{-m})\\\theta\in\Theta}}\big|S_t^mf(\varrho(y,t,\theta))\big|\Big\|_{L^q(\C_k(r))}\\
&\gtrsim
|\,D_1|
\left(
\sum_{j=0}^{2^k}\int_{\Omega_{k,j}}\,\d y
\right)^\frac1q\\
&\sim
(2^k)^\frac1q\lambda^{1-\frac1q}.
\end{align*}
Since $2^k=(r^{-k})^{\beta(\C(r))}=\lambda^{\beta(\C(r))}$, we would obtain $\lambda^{\frac12-\frac1q+\frac{\beta(\Theta)}{q}-s}\le C$  for some constant $C$. This is a contradiction as $\lambda\to\infty$.

For the remaining cases where $\beta(\Theta)=0$, $1$, one may let $\Theta=\{0\}$, $\mathbb I$, respectively, which are easily dealt with.  Indeed, the former coincides with the classical well-understood situation, and for the latter one can follow the argument above by setting $\theta(y)=y$ for all $y\in \mathbb I$.

Next, we shall consider the case when $d=2$ and basically modify the above argument. For the case of $\beta(\Theta)\in(0,1)$, we set $\Theta=\C(r)\times\{0\}$ and the initial data $f$ satisfying 
\begin{equation}\label{initial data necessary}
\widehat{f}(\xi)=e^{-i|\xi|^m}\chi_{D_2}(\xi),\quad D_2=[0,c\lambda]^2
\end{equation}
so that $\beta(\Theta)=\frac{\log2}{\log 1/r}\in (0,1)$ and $\|f\|_{H^s(\mathbb R^2)}\lesssim\lambda^s|\,D_2|^\frac12$. A similar change of variables gives that 
\[
\sup_{\substack{t\in \mathbb I\\\theta\in\Theta}}\big|S_t^m f(\varrho(-y,t,\theta))\big|
\geq
\left|
\int_{D_2} e^{i\phi(\xi,y,t,\theta)}
\right|,
\]
where 
\[
|\phi(\xi,y,t,\theta)|
=
|
(y_1-\theta_1(y))\xi_1+\tau\theta_1(y)\xi_1+(y_2-\theta_2(y))\xi_2+\tau\theta_2(y)\xi_2+\tau|\xi|^m
|.
\]
By choosing $y=(y_1,y_2)\in\C_k(r)\times[0,\lambda^{-1}]$ and $\tau\in(0,\lambda^{-m})$, we may have $|y_i-\theta_i(y)|<\lambda^{-1}$ for $i=1,2$, which guarantees that phase is small enough over $\xi\in D_2$, namely, $|\phi(\xi,y,t,\theta)|\leq\frac12$. Hence, 
\begin{align*}
\Big\|\sup_{\substack{t\in \mathbb I\\\theta\in\Theta}}\big|S_t^mf(\varrho(x,t,\theta))\big|\Big\|_{L^q(\mathbb B^2)}
&\gtrsim
\Big\|\sup_{\substack{\tau\in (0,\lambda^{-m})\\\theta\in\Theta}}\big|S_t^mf(\varrho(y,\tau,\theta))\big|\Big\|_{L^q(\C_k(r)\times[0,\lambda^{-1}])}\\
&\gtrsim
(2^k\lambda^{-2})^{\frac1q}|\,D_2|.
\end{align*}
Therefore, by assuming the sated estimate, we would obtain 
\[
(2^k\lambda^{-2})^{\frac1q}|\,D_2|\lesssim\lambda^s|\,D_2|^\frac12.
\]
By recalling $\lambda=r^{-k}$, $2^k=\lambda^{\beta(\Theta)}$ and $|D_2|\sim \lambda^2$, this yields a contradiction as $\lambda\to0$.

For the case when $\beta(\Theta)\in(1,2)$,  we set $\Theta=\C(r)\times[0,1]$ and employ the initial data given by \eqref{initial data necessary} so that the same computation reveals that the modulus of the corresponding phase $\phi(\xi,x,t,\theta)$ is bounded above by, say, $\frac12$ if we choose $y=(y_1,y_2)\in\C_k(r)\times[0,1]$ and $\tau\in (0, \lambda^{-m})$. Therefore, in this case, one may obtain
\[
(2^k\lambda^{-1})^{\frac1q}|\,D_2|\lesssim \lambda^s|\,D_2|^{\frac12},
\]
which result makes a contradiction.

The remaining are the cases where $\dim_M\Theta$ has the integers such as $0$, $1$, $2$ and may be considered similar to ones in the case of $d=1$. For instance, let $\Theta=\{0\}\times\{0\}$, $[0,1]\times\{0\}$, $[0,1]^2$, respectively. 

When $d\geq3$, a similar argument may be carried as well, which we omit the details.

\end{proof}

Theorem \ref{t:nec fractal path} encourages us to pursue further generalizations of Theorem \ref{t:max fractal} with respect to the $\alpha$-dimensional measure $\d \mu$.  By combining the argument in \cite{Sh19} with the one in \cite{CS20}, one may deduce the following.
\begin{theorem}\label{t:max fractal alpha-dim}
Let $d=1$, $m>1$ and $q\geq2$. If $s>\min\{\frac12,\max\{\frac14+\frac{\beta(\Theta)}{4\alpha},\frac12+\frac{\beta(\Theta)-\alpha}{q}\}\}$, then there exists $C>0$ such that 
\[
\|S_t^m f(\varrho(x,t,\theta))\|_{L_x^q(\mathbb I,\d\mu)L_t^\infty(\mathbb I)L_\theta^\infty(\Theta)}
\leq
C
\|f\|_{H^s(\mathbb R)}
\]
for all $f\in H^s(\mathbb R)$.
\end{theorem}

Note that the lower bound of $s$ is $\frac12$ if $\alpha<\beta(\Theta)$ and $\max\{\frac14+\frac{\beta(\Theta)}{4\alpha},\frac12+\frac{\beta(\Theta)-\alpha}{q}\}$ if $\beta(\Theta)\leq\alpha\leq 1$. Since the function $q\mapsto \frac12+\frac{\beta-\alpha}{q}$ is decreasing when $\alpha\geq\beta$, the choice of $q=2$ minimizes the lower bound of $s$ so that the poinwise convergence 
\[
\lim_{t\to0}S_t^m f(\varrho(x,t,\theta))=f(x),\quad \text{$\mu$-\ae}
\]
holds for all $f\in H^s$ whenever $s>\min\{\frac12,\max\{\frac14+\frac{\beta(\Theta)}{4\alpha}, \frac{1+\beta(\Theta)-\alpha}{2}\}\}$. Moreover, by Frostman's lemma one may find that $\sup_{f\in H^s}\dim_H \D(f\circ\varrho)\leq\{\frac{\beta(\Theta)}{4s-1},1+\beta(\Theta)-2s\}$ for $s\in [\frac{1+\beta(\Theta)}{4},\frac12]$ (see Figure \ref{f:dimH fractal}). 
\colorlet{Green}{green!70!black!}
\begin{figure}[h]
\begin{center}
\begin{tikzpicture}[scale=5]

\coordinate (O) at (0,0);
\draw [->](O)--(0,1.2);
\draw [->,name path=xaxis](O)--(2.2,0);

\node [above]at  (0,1.2) {$\alpha$};
\node [right]at  (2.2,0) {$s$};

\coordinate (Q) at (0.7,0);

\draw  [dotted](Q)--([yshift=1.2cm]Q);
\draw  [dotted,name path=v1, thick,]([xshift=0.7cm]Q)--([xshift=0.7cm,yshift=1.2cm]Q);

\draw [dotted, name path=h2] (0,1/2)--(2.2,1/2);

\path [name path=h1] (0,1)--(1,1);

\begin{scope}
\clip (0,0)--(2.2,0)--(2.2,1.2)--(0,1.2);
\draw [name path global=a,color=Green, domain=0.01:2, dashed,xshift=0.7cm] plot [samples=100](\x,1/\x*1/7);
\end{scope}

\path [name intersections={of= a and h2, by={M1}}];
\path [name intersections={of= a and v1, by={M2}}];
\path [name intersections={of= a and h1, by={M3}}];

\coordinate (c) at (1.8,0);

  \path let \p1=(M1), \p2=(M2), \p3=(c) in
  \pgfextra{\pgfmathsetmacro{\y}{(\y1-\y2)/(\x1-\x2)*(\x3-\x2)+\y2}}
   (c |- 0,\y pt) [name path=l]coordinate (d);
  \path [name path=l1] (M1)--(d);

\path [name intersections={of=l1 and xaxis, by={B}}]; 
    
\draw [red](M1)--(M2);
\draw [red,dashed] (M2)--(B);

\begin{scope}
\clip (M3|-O)--(M1|-O)--(M1)--(M3);
\draw [color=Green, domain=0.01:1,xshift=0.7cm] plot [samples=100](\x,1/\x*1/7);
\end{scope}   

\draw [dotted] (M1|-O)--([yshift=1.2cm]M1|-O);
\draw [dotted] (O|-M2)--([xshift=2.2cm]O|-M2);
\draw [blue] (0,1)--([yshift=1cm]Q);
\draw [dotted] ([yshift=1cm]Q)--(M3);
\draw [dotted] (M3|-O)--([yshift=1.2cm]M3|-O);

\node [below left] at (O) {$O$};
\node [left] at (O|-M2) {$\beta(\Theta)$};
\node [left] at (O|-M1) {$\frac12$};
\node [left] at (O|-M3) {$1$};
\node [below] at (Q) {$\frac14$};
\node [below] at (M3|-O) {$\frac{1+\beta(\Theta)}{4}$};
\node [below] at (M2|-O) {$\frac12$};
\node [below] at (B) {$\frac{1+\beta(\Theta)}{2}$};

\fill [white] ([yshift=1cm]Q) circle (0.5pt);
\draw ([yshift=1cm]Q) circle (0.5pt);
\fill [white] (M2) circle (0.5pt);
\draw (M2) circle (0.5pt);
\fill [white] (M3) circle (0.5pt);
\draw (M3) circle (0.5pt);

\end{tikzpicture}
\caption{The graph of $\alpha=\textcolor{blue}{1}$ when $s\in[0,\frac14)$ and $\alpha=\min\{\textcolor{Green}{\frac{\beta(\Theta)}{4s-1}},\textcolor{red}{1+\beta(\Theta)-2s}\}$ when $s\in(\frac{1+\beta(\Theta)}{4},\frac12)$.}\label{f:dimH fractal}
\end{center}
\end{figure}
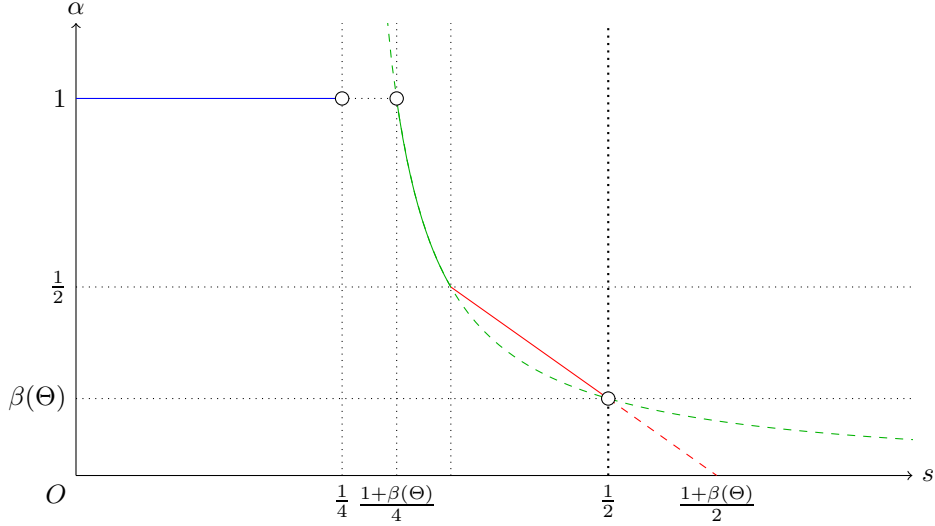
\begin{proof}
For a given $\Theta$, let us write $\beta$ for the shorthand for $\beta(\Theta)$.
Since $s>\frac12$ is obtained trivially, it is enough\footnote{As we discussed in above, the conclusion $s>\frac12$ when $\alpha<\beta$ also naturally appears by running an analogous argument carefully arranged for the case of $\alpha\geq\beta$.} to focus on the case $\alpha\geq\beta$.
Let $s_*=\min\{\frac14,\frac{\alpha}{q}\}$. 
The proof is based on a combination of arguments in \cite{Sh19} and \cite{CS20}. The following proposition has a crucial role in the proof, whose essential idea was introduced in \cite{Sh19}. 
\begin{proposition}\label{p:fractal suff}
Let
$\lambda\geq1$, $q\geq2$ and $\Omega$ be an interval of length $\lambda^{-\frac{qs_*}{\alpha}}$. For arbitrarily small $\varepsilon>0$, it holds that
\begin{equation}\label{e:fractal suff}
\Big\|\sup_{\substack{t\in[-1,1]\\ \theta\in\Omega}}\big|S_t^mf(\varrho(\cdot,t,\theta))\big|\Big\|_{L^q(\d\mu)}
\lesssim
\lambda^{\frac12-s_*+\varepsilon}\|f\|_{L^2}
\end{equation}
whenever $f$ is supported in $\{|\xi|\sim\lambda\}$. 
\end{proposition}
First of all, we assume that Proposition \ref{p:fractal suff} holds true and prove Theorem \ref{t:max fractal alpha-dim}. By using the dyadic decomposition with respect to frequency, we have
\[
\Big\|\sup_{\substack{t\in(-1,1)\\\theta\in\Theta}}\big|S_t^mf(\varrho(x,t,\theta))\big|\Big\|_{L^q}
\leq
\sum_{k\geq0}\Big\|\sup_{\substack{t\in(-1,1)\\\theta\in\Theta}}\big|S_t^mP_kf(\varrho(x,t,\theta))\big|\Big\|_{L^q}.
\]
For each $k$, by the definition of Minkowski dimension, one can find a finite collection of intervals $\{\Omega_{k,j}\}_{j=1}^{N_k}$ such that 
\[
\Theta\subset \bigcup_{j=1}^{N_k}\Omega_{k,j},\quad|\Omega_{k,j}|\leq(2^{-k})^{\frac{qs_*}{\alpha}}.
\]
Here, $N_\delta$ represents the smallest number of intervals of length $\delta$ that covers $\Theta$, and $N_k=N_{\delta_k}(\Theta)$ with $\delta_k=(2^{-k})^{\frac{qs_*}{\alpha}}$ in particular.
Thus, for each $k$,
\[
\sup_{\substack{t\in(-1,1)\\\theta\in\Theta}}\big|S_t^mP_kf(\varrho(x,t,\theta))\big|^q
\leq
\sum_{j=1}^{N_k}\sup_{\substack{t\in(-1,1)\\\theta\in\Omega_{k,j}}}\big|S_t^mP_kf(\varrho(x,t,\theta))\big|^q,
\]
from which it follows that
\[
\Big\|\sup_{\substack{t\in(-1,1)\\\theta\in\Theta}}\big|S_t^mP_kf(\varrho(x,t,\theta))\big|\Big\|_{L^q}
\leq
\left(
\sum_{j=1}^{N_k}\Big\|\sup_{\substack{t\in(-1,1)\\\theta\in\Omega_{k,j}}}|S_t^mP_kf(\varrho(x,t,\theta))\big|\Big\|_{L^q}^q
\right)^\frac1q.
\]
Invoking Proposition \ref{p:fractal suff} with $\lambda=2^k$, and the fact $N_k\lesssim(2^k)^{\frac{qs_*}{\alpha}\beta+\epsilon}$ for any $\epsilon>0$ (from the definition of Minkowski dimension), we obtain
\begin{align*}
\Big\|\sup_{\substack{t\in(-1,1)\\\theta\in\Theta}}\big|S_t^mf(\varrho(x,t,\theta))\big|\Big\|_{L^q}
&\leq
\sum_{k\geq0}
\left(
\sum_{j=1}^{N_k}\Big\|\sup_{\substack{t\in(-1,1)\\\theta\in\Omega_{k,j}}}\big|S_t^mP_kf(\varrho(\cdot,t,\theta))\big|\Big\|_{L^q}^q
\right)^\frac1q\\
&\lesssim
\|P_0f\|_{L^2}+\sum_{k\geq1}
\left(
\sum_{j=1}^{N_k}
\left[
(2^k)^{\frac12-s_*+\varepsilon}\|P_kf\|_{L^2}
\right]^q
\right)^\frac1q\\
&\lesssim
\|P_0f\|_{L^2}+\sum_{k\geq1}
(2^k)^{\frac12-s_*+\varepsilon+\frac{s_*}{\alpha}\beta}\|P_kf\|_{L^2}\\
&\lesssim
\|f\|_{H^{\frac12-(1-\frac{\beta}{\alpha})s_*+\varepsilon}},
\end{align*}
as aimed. \\

Now, we turn to the proof of the proposition.
Let the operator $T$ on $L^2$ be given by
\[
Tf(x,t,\theta)=\chi(x,t,\theta)\int_{\mathbb{R}}e^{i(\varrho(x,t,\theta)\xi+t|\xi|^m)}f(\xi)\psi(\tfrac\xi\lambda)\,\d\xi,
\]
where $\chi=\chi_{I\times I\times \Omega}$ and $\psi\in C_0^\infty((-2,-\frac12)\cup(\frac12,2))$. Then, by the Plancherel theorem and duality, \eqref{e:fractal suff} is equivalent to 
\begin{equation}\label{d:fractal suff}
\|T^*F\|^2_{L^2}\lesssim\lambda^{1-2s_*+\varepsilon}\|F\|_{L^{q'}_x(\d\mu)L^1_tL_\theta^1}^2,
\end{equation}
where $T^*$ is the adjoint of $T$.  
To estimate \eqref{d:fractal suff} we decompose $\|T^*F\|_{L^2}$ into $\mathcal I_1,\mathcal I_2,\mathcal I_3$ such that
\begin{align*}
\|T^*F\|_{L^2}^2
=
\mathcal I_1+\mathcal I_2+\mathcal I_3,
\end{align*}
where, under the notations $W=I\times I\times \Omega$, $w=(x,t,\theta)\in W$, $w'=(x',t',\theta')\in W$ and $\d_\mu w=\d\mu(x)\d t\d \theta$, we set
\[
\mathcal I_\ell:=\iint_{V_\ell}\chi(w)\chi(w')\bar{F}(w)F(w')K_\lambda(w,w')\,\d_\mu w\,\d_\mu w',
\]
\begin{align*}
K_\lambda(w,w')
:=
\int_\mathbb{R} e^{i\phi(\xi,w,w')}\psi(\tfrac{\xi}{\lambda})^2\,\d\xi
=
\lambda\int_\mathbb{R} e^{i\phi(\lambda\xi,w,w')}\psi(\xi)^2\,\d\xi,
\end{align*}
\[
\phi(\xi,w,w')=(\varrho(x,t,\theta)-\varrho(x',t',\theta'))\xi+(t-t')|\xi|^m
\]
and
\[
\begin{cases}
V_1=\{(w,w')\in W\times W:{|x-x'|\le2\lambda^{-\frac{q s_*}{\alpha}}}\},\\
V_2=\{(w,w')\in W\times W: |x-x'|>2\lambda^{-\frac{q s_*}{\alpha}}\ {\rm and}\  |x-x'|\le 4|t-t'|\},\\
V_3=\{(w,w')\in W\times W:|x-x'|>2\lambda^{-\frac{qs_*}{\alpha}}\ {\rm and}\ |x-x'|> 4|t-t'|\}.
\end{cases}
\]
Therefore, it is enough to show that for each $\ell=1,2,3$
\[
\mathcal I_\ell\lesssim \lambda^{1-2s_*+\varepsilon}\|F\|_{L^{q'}_x(\d\mu)L^1_tL^1_\theta}^2.
\]

The case when $\ell=1$ immediately follows from the trivial kernel estimate $|K_\lambda(w,w')|\lesssim\lambda$ and Lemma \ref{l:Young and HLS-type}. 

For $\I_2$, we shall observe that 
\begin{align}\label{second derivative}
\left|\frac{\d^2}{\d\xi^2}\phi(\lambda\xi,w,w')\right|
\gtrsim
\lambda^m|t-t'||\xi|^{m-2}
\gtrsim
\lambda|x-x'| \gtrsim \lambda^{1-\frac{qs_*}{\alpha}}
\ge
1
\end{align}
because of $\frac{qs_*}{\alpha}\le1$. One can apply van der Corput's lemma (Lemma \ref{lem:van der Corput}) to get
\begin{align*}
|K_\lambda(w,w')|
\lesssim
\lambda(\lambda|x-x'|)^{-\frac12}
\lesssim
\lambda(\lambda|x-x'|)^{-2s_*}
\lesssim
\textcolor{black}{\lambda^{1-2s_*+\varepsilon}|x-x'|^{-2s_*+\varepsilon}}
\end{align*}
from the separation assumption and the fact $2s_*\le\frac12$.
Therefore, applying Lemma \ref{l:Young and HLS-type} with $\rho=2s_*-\varepsilon<\frac{2\alpha}{q}$, it follows that
\[
\mathcal I_2\lesssim\lambda^{1-2s_*+\varepsilon}\|F\|_{L^{q'}_x(\d\mu)L^1_t}^2.
\]

Finally, for $\I_3$, note a key relation
\begin{equation}\label{i:gamma-gamma}
\textcolor{black}{|\varrho(w)-\varrho(w')|\sim|x-x'|}
\end{equation}
for $(w,w')\in V_3$ since
\begin{align}\label{i:compare theta and x}
|\theta-\theta'|\leq\lambda^{-\frac {qs_*}{\alpha}}<\frac12|x-x'|.
\end{align}

We split $K_\lambda$ into $\mathcal K_1$ and $\mathcal K_2$ as follows.
\begin{align*}
K_\lambda(w,w')&=\lambda\int_{U_1}e^{i\phi(\lambda\xi,w,w')}\psi(\xi)^2\,\d\xi+\lambda\int_{U_2}e^{i\phi(\lambda\xi,w,w')}\psi(\xi)^2\,\d\xi\\
&=:\mathcal K_1+\mathcal K_2,
\end{align*}
where 
\[
\begin{cases}
U_1=\{\xi\in\supp\psi:|x-x'|>8m\lambda^{m-1}|t-t'||\xi|^{m-1}\},\\
U_2=\{\xi\in\supp\psi:|x-x'|\le8m\lambda^{m-1}|t-t'||\xi|^{m-1}\}.
\end{cases}
\]
For $\mathcal K_1$,
we use \eqref{i:gamma-gamma} and $\frac{qs_*}{\alpha}\le1$ to estimate
\begin{align*}
\left|\frac{\d}{\d\xi}\phi(\lambda\xi,w,w')\right|&\ge\big|\lambda|\varrho(w)-\varrho(w')|-m\lambda^m|t-t'||\xi|^{m-1}\big|\\
&\gtrsim|\lambda|x-x'|-m\lambda^m|t-t'||\xi|^{m-1}|\\
&\gtrsim\lambda|x-x'|\gtrsim \lambda^{1-\frac{qs_*}{\alpha}}
\ge1.
\end{align*}
Notice that the interval $U_1$ consists of at most two intervals since $\frac{\d}{\d\xi}\phi(\lambda\xi,w,w')$ is monotone on each interval $(-\infty,-1]$ 
and $[1,\infty)$. Then, van der Corput's lemma  gives that 
\begin{align}\label{$K_1$estimate} 
\mathcal K_1
\lesssim
\lambda(\lambda|x-x'|)^{-1}.
\end{align}
On the other hand, for $\mathcal K_2$, we have \eqref{second derivative} again and apply van der Corput's lemma to obtain
\begin{align}\label{$K_2$estimate}
\mathcal K_2&\lesssim\lambda(\lambda|x-x'|)^{-\frac12}.
\end{align}

Combining \eqref{$K_1$estimate} and \eqref{$K_2$estimate}, for $(w,w')\in V_3$
\begin{align*}
|K_\lambda(w,w')|&\lesssim\lambda(\lambda|x-x'|)^{-\frac12}\lesssim\lambda(\lambda|x-x'|)^{-2s_*}\\
&\lesssim\lambda^{1-2s_*+\varepsilon}|x-x'|^{-2s_*+\varepsilon}
\end{align*}
from the separation assumption.
By Lemma \ref{l:Young and HLS-type} with $\rho=2s_*-\varepsilon$ we conclude that
\begin{align*}
\mathcal I_3
&\lesssim
\lambda^{1-2s_*+\varepsilon}\|F\|_{L^{q'}_x(\d\mu)L^1_tL^1_\theta}^2.
\end{align*}
\end{proof}

\begin{remark}
%
It is, of course, reasonable to generalize Theorem \ref{t:nec fractal path} in the context of $\alpha$-dimensional measure and hope that $s\geq\frac12+\frac{\beta-\alpha}{q}$ is also necessary. We found, however, that this may not be straightforward and note here that a weaker condition, $s\geq\frac12+\frac{\beta-1}{q}$, is necessary for $\alpha\in(0,1]$ when $d=1$: For $\alpha\in(0,1]$ by employing $\d\mu(x)=|x|^{\alpha-1}\d x$, instead of $\d x$, one may reach to
\begin{align*}
\|S_t^mf(\varrho(x,t,\theta)\|_{L_x^q(\mathbb I,\,\d\mu)L_t^\infty(\mathbb I)L_\theta^\infty(\Theta)}
&\gtrsim
\lambda\left(\sum_{j=1}^{2^k}\int_{\Omega_{k,j}}\,\d\mu(x)\right)^\frac1q\\
&\sim
\lambda\left(\sum_{j=1}^{2^k} \big[(y_{k,j}+\frac1\lambda)^\alpha-y_{k,j}^\alpha\big] \right)^\frac1q.
\end{align*}
Here, $(y_{k,j})_{j}$ represents the left-end points of intervals consisting of $k$-th generation of pre-Cantor set $\C_k(r)$; i.e.  $\Omega_{k,j}=[y_{k,j},y_{k,j}+\frac1\lambda]$. Then, the mean value theorem gives that 
\[
(y_{k,j}+\frac1\lambda)^\alpha-y_{k,j}^\alpha\gtrsim\frac1\lambda,
\]
which clearly gives what we claimed.

\end{remark}

There are many other variations of the maximal inequality for the (fractional) Schr\"odinger equation. In the classical higher dimensional cases, the maximal inequality \eqref{i:max classical} for radial initial data was considered by Prestini \cite{Pr90}, and later its fractal dimension of the divergence sets was computed by Bennett--Rogers \cite{BR12}. 
Some results when $m\in(0,1]$ are also known but appeared to possess a rather different nature than when $m>1$ (the reader may visit \cite{CS22, Cw82, HKL21, JY21, RV08, Wl95}).  The periodic setting (replacing $\mathbb R^d$ by the torus $\mathbb T^d$) is also intriguing. Moyua--Vega \cite{MV08} considered the one-dimensional case and obtained some sufficient conditions and necessary conditions, although there is still a gap between them remaining.  For the results in the higher dimensions, for example, see work by Wang--Zhang \cite{WZ19}, Eceizabarrena--Luc\`a \cite{EL20} as well as Compaan--Luc\`a--Staffilani \cite{CLS20}, where they also discuss the pointwise convergence problem for the solution to certain nonlinear Schr\"odinger equation.
Another interesting variation is due to Bez--Lee--Nakamura \cite{BLN19}; the local maximal inequality for orthonormal systems of initial data $(f_j)_j$. Their results in one spatial dimension teach the pointwise convergence behavior of finitely many fermion particles interacting with each other on a line. 
Bez--Kinoshita and the second author further computed the Hausdorff dimensions of the corresponding divergence sets as well \cite{BKS}.
Recently, Dimou and Seeger \cite{DS20} considered pointwise convergence problem along a time sequence that rapidly approaches to $0$. It turns out that such sequential convergence may require less smooth regularity than the original convergence. 
They obtained sharp results for the fractional Schr\"odinger equations in one dimension, which were extended to higher dimensions by Sj\"olin \cite{Sj19}, Sj\"olin and Str\"omberg \cite{SS20, SS21a, SS21b} and Li, Wang, and Yan \cite{LWY20b}. Later, Li, Wang, and Yan \cite{LWY22} and Ko, Koh, Lee and first author \cite{CKKL22} established sharp results for the fractional Schr\"odinger equations  and more general dispersive equations defined in higher dimensions by using spacial localization.

\begin{appendix}
\section{}\label{a:Cantor sets}

Recall the $r$-Cantor set whose construction was given in Section \ref{s:fractal}.
\begin{lemma}
For $r\in(0,\frac12)$,
\[
\dim_M \C(r)=\frac{\log 2}{\log1/r},
\]
ranged in $(0,1)$.
\end{lemma}
\begin{proof}
Since $\C_k(r)$ consists of $2^k$ disjoint intervals of length $r^k$, we have $N_{r^k}(\C(r))\leq N_{r^k}(\C_k(r))=2^k$ from which it follows that
\begin{align*}
\limsup_{k\to \infty}\frac{\log N_{r^k}(\C(r))}{-\log1/r^k}
\leq
\limsup_{k\to\infty}\frac{\log 2^k}{\log1/r^k}
=
\frac{\log 2}{\log1/r}.
\end{align*}
On the other hand, by recalling the construction, one can find, at least, a point in $\C_k(r)\cap \C(r)$ which is not covered by any $2^k$ intervals, where the length of each interval is $\delta$ satisfying $r^{k+1}\leq\delta < r^{k}$.
Therefore, $N_\delta(\C(r))\geq 2^k$ holds, and 
\begin{align*}
\liminf_{\delta\to 0}\frac{\log N_\delta(\C(r))}{-\log \delta}
\geq
\liminf_{k\to\infty}\frac{\log 2^k}{\log1/r^{k+1}}
=
\frac{\log 2}{\log 1/r}.
\end{align*}
Hence, the limit exists and equals
\begin{equation}
\dim_M \C(r)
=
\lim_{\delta\to 0}\frac{\log N_\delta(\C(r))}{-\log \delta}
=
\frac{\log 2}{\log 1/r}.
\end{equation}

\end{proof}

\section{}
In this section, we note the useful lemmas for the sufficiency in Section \ref{s:fractal}.

\begin{lemma}[van der Corput's lemma]\label{lem:van der Corput}
Let $a,b$ be real numbers with $a<b$, $\phi$ be a sufficiently smooth real-valued function, and $\psi$ be a bounded smooth complex-valued function. Suppose that $|\phi^{(k)}(\xi)|\geq1$ for all $\xi \in [a,b]$. If $k=1$ and $\phi'(\xi)$ is monotonic on $(a,b)$, or simply $k\geq2$, then there exists a constant $C_k$ such that
\[
\left|\int_a^be^{i\lambda\phi(\xi)}\psi(\xi)\,\d\xi\right|\leq C_k\lambda^{-\frac1k}\left(\|\psi'\|_{L^{1}[a,b]}+\|\psi\|_{L^{\infty}[a,b]}\right)
\]
for all $\lambda > 0$.
\end{lemma}
For a proof of Lemma \ref{lem:van der Corput}, we refer the reader to \cite{St94}. Next we introduce a Young convolution/Hardy--Littlewood--Sobolev type-inequalities, which generalizes lemmas in \cite{Sh19,CS20}.

\begin{lemma}\label{l:Young and HLS-type}
Let $0<\alpha\le1$, \textcolor{black}{$q\geq2$} and $\mu$ be an $\alpha$-dimensional measure. There exists a constant $C$ such that for any interval $[a,b]$ in $\mathbb R$
it holds that
\begin{align}\label{i:Young special case}
&\left|\iint\iint g(x,t)h(x',t')\chi_{[a,b]}(x-x')\,\d \mu(x)\d t\d \mu(x')\d t'\right|\\
&\qquad \qquad \qquad \qquad \qquad \qquad \qquad \qquad \qquad\le C(b-a)^{\frac{2\alpha}{q}}\|g\|_{L^{q'}_x(\d\mu)L^1_t}\|h\|_{L^{q'}_x(\d\mu)L^1_t}.\nonumber
\end{align}
Moreover, if \textcolor{black}{$0<\frac{q\rho}{2}<\alpha$} then there exists a constant $C$ such that 
\begin{align}\label{i:HLS-type}
&\left|\iint\iint g(x,t)h(x',t')|x-x'|^{-\rho}\,\d \mu(x)\d t\d \mu(x')\d t'\right| \le C\|g\|_{L^{q'}_x(\d\mu)L^1_t}\|h\|_{L^{q'}_x(\d\mu)L^1_t}.
\end{align}
Here, the both integrals are taken over $(x,t)$, $(x',t')\in \mathbb I\times \mathbb I$.
\end{lemma}
This lemma is a generalization of Lemma 4 in \cite{Sh19} and Lemma 7 in \cite{CS20}.

\begin{proof}
In order to show \eqref{i:Young special case}, it is enough to show that 
\begin{equation*}
\|g\ast_\mu\chi_{[a,b]}\|_{L^q(\d\mu)}\lesssim(b-a)^{\frac{2\alpha}{q}}\|g\|_{L^{q'}(\d\mu)}.
\end{equation*}
By applying H\"older's inequality with $\frac1q+\frac{1}{q'}=1$,
\begin{align*}
&\left(\int\left|\int g(x')\chi_{[a,b]}(x-x')\,\d\mu(x')\right|^q\,\d\mu(x)\right)^\frac1q\\
& \qquad \qquad \leq
(b-a)^\frac\alpha q\left(\int\left|\int |g(x')|^{q'}\chi_{[a,b]}(x-x')\,\d\mu(x')\right|^\frac{q}{q'}\,\d\mu(x)\right)^\frac1q,
\end{align*}
which is further bounded, as a result of Minkowski's inequality since $\frac{q}{q'}\geq1$, from above by 
\begin{align*}
(b-a)^\frac\alpha q\left(\int\left(\int |g(x')|^{\textcolor{black}{q}}\chi_{[a,b]}(x-x')\,\d\mu(x)\right)^{\textcolor{black}{\frac{q'}{q}}}\,\d\mu(x')\right)^\frac{1}{q'}
\sim
(b-a)^\frac{2\alpha}{q}\|g\|_{L^{q'}(\d\mu)}.
\end{align*}

Then, showing \eqref{i:HLS-type} is rather easy via \eqref{i:Young special case} as follows:
\begin{align*}
&\left|\iint\iint g(x,t)h(x',t')|x-x'|^{-\rho}\,\d \mu(x)\d t\d \mu(x')\d t'\right|\\
& \qquad \qquad \lesssim\sum_{j=0}^\infty2^{\rho j}\iint G(x)H(x')\chi_{[2^{-j},2^{-j+1}]}(x-x')\,\d\mu(x)\d\mu(x')\\
& \qquad \qquad\lesssim\sum_{j=0}^\infty 2^{(\rho-\frac{2\alpha}{q})j}\|G\|_{L^{q'}_x(\d\mu)}\|H\|_{L^{q'}_x(\d\mu)}\\
& \qquad\qquad\lesssim\|G\|_{L^{q'}_x(\d\mu)}\|H\|_{L^{q'}_x(\d\mu)}
\end{align*}
whenever $\rho-\frac{2\alpha}{q}<0$.
\end{proof}

\end{appendix}

\subsection*{Acknowledgment} 
The authors would like to thank Satoshi Masaki and Yutaka Terasawa for the great opportunity at RIMS in the summer of 2021 and their kind hospitality. The second author also wishes to thank Neal Bez, Haruya Mizutani, and Shohei Nakamura for several constructive discussions related to some of this work. Finally,
we thank the anonymous referee for their helpful suggestions.



\begin{thebibliography}{99}
\bibitem{ACP21}
C. An, 
R. Chu, 
L. B. Pierce,
\textit{Counterexamples for high-degree generalizations of the Schr\"odinger maximal operator},
arXiv:2103.15003.


\bibitem{BBCR11}
J. A. Barcel\'o,
J. Bennett,
A. Carbery,
K. M. Rogers,
\textit{On the dimension of divergence sets of dispersive equations},
Math. Ann. \textbf{348} (2011), 599--622.

 
\bibitem{BR12}
J. Bennett,
K. M. Rogers,
\textit{On the size of divergence sets for the Schr\"odinger equation with radial data},
Indiana Univ. Math. J. {\bf 61} (2012), 1--13.	 

\bibitem{BKS}
N. Bez,
S. Kinoshita,
S. Shiraki,
in preparation.


\bibitem{BLN19}
N. Bez, 
S. Lee, 
S. Nakamura, 
\textit{Maximal estimates for the Schrödinger equation with orthonormal initial data}, 
Selecta Math. \textbf{26} (2020), article number 52.




\bibitem{Br95} J. Bourgain, 
	\textit{Some new estimates on oscillatory integrals}, 
	Essays on Fourier analysis in honor of Elias M. Stein (Princeton, NJ, 1991)
	(1995), 83--112.
	 
\bibitem{Br13}
J. Bourgain,
\textit{On the Schr\"odinger maximal function in higher dimension}, 
Tr. Mat. Inst. Steklova \textbf{280} (2013), 46--60.


\bibitem{Br16}
J. Bourgain,
\textit{A note on the Schr\"odinger maximal function},
J. Anal. Math. \textbf{130} (2016), 393--396.



\bibitem{Cr80}
L. Carleson,
\textit{Some analytic problems related to statistical mechanics},
In: Euclidean harmonic analysis (Proc. Sem., Univ. Maryland, College Park, Md., 1979), Lecture Notes in Math. \textbf{779} (1980), 5--45.    

\bibitem{CK18}
C. H. Cho,
H. Ko,
\textit{Pointwise convergence of the fractional Schrödinger equation in $\mathbb R^2$}, Taiwanese J. Math. \textbf{26} (2022), 177--200.

\bibitem{CKKL22}
C. H. Cho,
H. Ko,
Y.Koh,
S. Lee,
\textit{Pointwise convergence of sequential Schrödinger means},
arXiv:2207.09219.

\bibitem{CKL21}
C.H.Cho,
Y.Koh,
J.Lee,
\textit{A global space-time estimate for dispersive operators through its local estimate},
arXiv:2109.00382.

\bibitem{CKS16}
C. H. Cho,
Y. Koh,
I. Seo,
\textit{On inhomogeneous Strichartz estimates for fractional Schr\"odinger equations and their applications},
Discrete Contin. Dyn. Syst. \textbf{36} (2016), 1905--1926.    


\bibitem{CL14}
C. H. Cho,
S. Lee,
\textit{Dimension of divergence sets for the pointwise convergence of the Schr\"odinger equation},
J. Math. Anal. Appl. \textbf{411} (2014), 254--260. 


\bibitem{CLV12}
C. H. Cho,
S. Lee,
A. Vargas,
\textit{Problems on pointwise convergence of solutions to the Schr\"odinger equation},
J. Fourier Anal. Appl. \textbf{18} (2012), 972--994. 


\bibitem{CS20}
C. H. Cho, 
S. Shiraki,
\textit{Pointwise convergence along a tangential curve for the fractional Schr\"odinger equations},
Ann Fenn Math. \textbf{46} (2021), 993--1005.
     
\bibitem{CS22}
C. H. Cho, 
S. Shiraki,
\textit{Dimension of divergence sets of oscillatory integrals with concave phase},
preprint.     
     


\bibitem{Cw82}
M. G. Cowling,
\textit{Pointwise behavior of solutions to Schr\"odinger equations},
Harmonic Analysis (Cortona, 1982), Lecture Notes in Math. \textbf{992} (1983), 83--90.

\bibitem{CLS20}
E. Compaan,
R. Luc\`{a},
G. Staffilani,
\textit{Pointwise convergence of the Schr\"odinger flow},
to appear in Int. Math. Res. Not. 

\bibitem{DK82}
B. E. J. Dahlberg,
C. E. Kenig, 
\textit{A note on the almost everywhere behavior of solutions to the Schr\"odinger equation}. 
In: Harmonic analysis (Minneapolis, Minn., 1981), Lecture Notes in Math. \textbf{908} (1982), 205--209.
    


\bibitem{DS20} 
E. Dimou, 
A. Seeger,
\textit{On pointwise convergence of Schr\"odinger means}, Mathematika. 66, 356-372 (2020).


    
\bibitem{DGL17}
X. Du,
L. Guth,
X. Li,
\textit{A sharp Schr\"odinger maximal estimate in $\mathbb{R}^2$},
Ann. of Math. \textbf{186} (2017), 607--640.
   

\bibitem{DGLZ18}   
X. Du,
L. Guth,
X. Li,
R. Zhang,
\textit{Pointwise convergence of Schr\"odinger solutions and multilinear refined Strichartz estimates}
Forum Math. Sigma \textbf{6} (2018), 18 pp.

\bibitem{DZ18}
X. Du,
R. Zhang,
\textit{Sharp $L^2$ estimate of Schr\"odinger maximal function in higher dimensions},
Ann. of Math. \textbf{189} (2019), 837--861.

\bibitem{Duo}
J. Duoandikoetxea, 
\textit{Fourier Analysis}, 
Grad. Studies Math. \textbf{29}, Amer. Math. Soc., Providence,2000.

\bibitem{EL20}
D. Eceizabarrena, R. Luc\`{a},
\textit{Convergence over fractals for the periodic Schr\"odinger equation},
arXiv:2005.07581.

\bibitem{EP22}
D. Eceizabarrena, F. Ponce-Vanegas,
\textit{Pointwise convergence over fractals for dispersive equations with homogeneous symbol},
J. Math. Anal. Appl.
\textbf{515}
(2022),
126385.

    
 

 

\bibitem{HKL21}
 S. Ham, 
 H. Ko, 
 S. Lee, 
 \textit{Dimension of divergence set of the wave equation}, 
 Nonlinear Anal. \textbf{215} (2022), 112631.
 



\bibitem{JY21}
T. Jhao, J. Yuan, 
\textit{Pointwise convergence along a tangential curve for the fractional Schr\"odinger equation with $0<m<1$},
to appear in Math. Methods Appl. Sci.


\bibitem{Jo10}
K.  Johansson,
\textit{A counterexample on nontangential convergence for oscillatory integrals},
Publ. Inst. Math. (Beograd) (N.S.) \textbf{87} (2010), 129--137.
  

    
\bibitem{KPV91}
C. E. Kenig,
G. Ponce,
L. Vega,
\textit{Oscillatory integrals and regularity of dispersive equations},
Indiana Univ. Math. J. \textbf{40} (1991), 33--69.
    
%
    
\bibitem{Lee06}
S. Lee,
\textit{On pointwise convergence of the solutions to Schr\"odinger equations in $\mathbb{R}^2$},
Int. Math. Res. Not. (2006), 1--21.

\bibitem{LR12}
S. Lee,
K. Rogers,
\textit{The Schr\"odinger equation along curves and the quantum harmonic oscillator},
Adv. Math. \textbf{229} (2012), 1359--1379.


\bibitem{LW18} 
W. Li,  
H. Wang, 
\textit{Pointwise convergence of solutions to the Schrödinger equation along a class of curves}, 
arXiv: 1807.00292.

\bibitem{LW21} 
W. Li,  
H. Wang, 
\textit{On convergence properties for generalized Schrödinger equation along tangential curves},
arXiv:2111.09186.


\bibitem{LWY20} 
W. Li, 
H. Wang, 
D. Yan, 
\textit{A note on Non-tangential Convergence for Schrödinger operators}, 
arXiv:2008.03093.

\bibitem{LWY20b} 
W. Li, 
H. Wang, 
D. Yan, 
\textit{Pointwise Convergence for sequences of Schr\"odinger means in $\mathbb R^2$},
 arXiv:2010.08701.

\bibitem{LWY22} 
W. Li, 
H. Wang, 
D. Yan, 
\textit{Sharp convergence for sequences of Schrödinger means and related generalizations},
 arXiv:2207.08440.

\bibitem{LP21}
R. Luc\`{a},
P. Ponce-Vanegas,
\textit{Convergence over fractals for the Schr\"odinger equation}, 
arXiv:2101.02495.

\bibitem{LR17}
R. Luc\`{a},
K. M. Rogers,
\textit{Coherence on fractals versus pointwise convergence for the Schr\"odinger equation},
Comm. Math. Phys. \textbf{351} (2017), 341--359.


\bibitem{LR19}
R. Luc\`{a},
K. M. Rogers,
\textit{A note on pointwise convergence for the Schr\"odinger equation},
Math. Proc. Cambridge Philos. Soc. \textbf{166} (2019), 209--218.

\bibitem{LR19b}
R. Luc\`{a},
K. M.  Rogers, 
\textit{Average decay of the Fourier transform of measures with applications},
J. Eur. Math. Soc. (JEMS) \textbf{21} (2019), 465--506.




\bibitem{MYZ15} C. Miao, J. Yang, J. Zheng, 
\textit{An improved maximal inequality for $2D$ fractional order Schr\"odinger operators},
Studia Mathematica. \textbf{230} (2015), 121--165.
	

\bibitem{MV08}
A. Moyua, L. Vega. 
\textit{Bounds for the maximal function associated to periodic solutions of one-dimensional dispersive equations},
Bull. Lond. Math. Soc. \textbf{40} (2008), 117--128.








\bibitem{Pierce}
L. B. Pierce, 
\textit{On Bourgain’s Counterexample for the Schr\"{o}dinger Maximal Function}, 
Q. J. Math. \textbf{71} (2020), 1309--1344.

\bibitem{Pl02}
F. Planchon,
\textit{Dispersive estimates and the 2D cubic NLS equation},
J. Anal. Math. \textbf{86} (2002), 319--334.


\bibitem{Pr90}
E. Prestini,
\textit{Radial functions and regularity of solutions to the Schr\"odinger equation},
Monatsh. Math. \textbf{109} (1990), 135--143.



\bibitem{RV08}
K. M. Rogers,
P. Villarroya,
\textit{Sharp estimates for maximal operators associated to the wave equation},
Ark. Mat. \textbf{46} (2008), 143--151. 

\bibitem{Sh19}
S. Shiraki,
\textit{Pointwise convergence along restricted directions for the fractional Schr\"odinger equation}, 
J. Fourier Anal. Appl. \textbf{26} (2020), 12 pp.
    
\bibitem{SS89}
P. Sj\"ogren,
P. Sj\"olin,
\textit{Convergence properties for the time dependent Schr\"odinger equation},
Ann. Acad. Sci. Fenn. A I Math. \textbf{14} (1989), 13--25.
   
\bibitem{Sj87}
P. Sj\"olin,
\textit{Regularity of solutions to the Schr\"odinger equation},
Duke Math. J. \textbf{55} (1987), 699--715.

\bibitem{Sj19}
P. Sj\"olin,
\textit{Two Theorems on Convergence of Schr\"odinger Means}, J. Fourier Anal. Appl. 25(4), 1708--1716 (2019).

\bibitem{SS20} 
P. Sj\"olin, 
J-O. Str\"omberg, 
\textit{Convergence of sequences of Schr\"odinger means}, J. Math. Anal. Appl. 483,123580, 23 pp  (2020).

\bibitem{SS21a} 
P. Sj\"olin, 
J-O. Str\"omberg, 
\textit{Schrödinger means in higher dimensions}, J. Math. Anal. Appl. 504, 125353, 32 pp (2021).

\bibitem{SS21b} 
P. Sj\"olin, 
J-O. Str\"omberg, 
\textit{Analysis of Schrödinger means},  Ann. Fenn. Math. 46, 389--394 (2021).

\bibitem{St}
E. M. Stein, 
\textit{On limits of sequences of operators},
Ann. Math. \textbf{74} (1961), 140--170.
    
\bibitem{St94}
E. M. Stein,
\textit{Harmonic Analysis},
Princeton Univ. Press, Princeton (1993).


\bibitem{Vg88}
L. Vega,
\textit{Schr\"odinger equations: pointwise convergence to the initial data},
Proc. Amer. Math. Soc. \textbf{102} (1988), 874--878. 
    
    
    
\bibitem{Wl95}
B. G. Walther,
\textit{Maximal estimates for oscillatory integrals with concave phase},
Contemp. Math. \textbf{189} (1995), 485--495.

\bibitem{WZ19}
X. Wang, C. Zhang,
\textit{Pointwise convergence of solutions to the Schr\"odinger equation on manifolds},
Canad. J. Math.  \textbf{71} (2019), 983--995. 


\bibitem{Zb02}
D. \v Zubrini\'c,
\textit{Singular sets of Sobolev functions},
C. R. Math. Acad. Sci. Paris  \textbf{334} (2002), 539--544.

 \end{thebibliography}
\end{document}